\documentclass[12pt]{amsart}

\usepackage{amssymb}

\usepackage{bm}
\usepackage{graphicx}
\usepackage[centertags]{amsmath}
\usepackage{amsfonts}
\usepackage{amsthm}
\usepackage{enumerate}

\usepackage{xcolor}
\usepackage[margin=1in]{geometry}

\newtheorem{theorem}{Theorem}
\newtheorem*{theorem*}{Theorem}

\newtheorem{lemma}[theorem]{Lemma}

\theoremstyle{definition}

\newenvironment{remark}[1][Remark.]{\begin{trivlist}
\item[\hskip \labelsep {\bfseries #1}]}{\end{trivlist}}
\newenvironment{remarks}[1][Remarks.]{\begin{trivlist}
\item[\hskip \labelsep {\bfseries #1}]}{\end{trivlist}}

\newcommand{\infin}{[\mathbb{N}]}

\newcommand{\nn}{\mathbb{N}}

\newcommand{\ee}{\varepsilon}

\newcommand{\vvv}{\mathcal{V}}

\newcommand{\uuu}{\mathcal{U}}

\begin{document}

\title{Almost isometric constants for partial unconditionality}
\author{R. M. Causey}

\address{Department of Mathematics\\ University of South Carolina\\
Columbia, SC 29208\\ U.S.A.} \email{causey@math.sc.edu}

\author{S. J. Dilworth}

\address{Department of Mathematics\\ University of South Carolina\\
Columbia, SC 29208\\ U.S.A.} \email{dilworth@math.sc.edu}

\begin{abstract} We discuss optimal constants of certain projections on subsequences of weakly null sequences.  Positive results yield constants arbitrarily close to $1$ for Schreier type projections, and arbitrarily close to $1$ for Elton type projections under the assumption that the weakly null sequence admits no subsequence generating a $c_0$ spreading model.  As an application, we prove that a weakly null sequence admitting a spreading model not equivalent to the $c_0$ basis has a quasi-greedy subsequence with quasi-greedy constant arbitrarily close to $1$.

\end{abstract}

\thanks{The second author was supported by the National Science Foundation under Grant Number DMS-1361461}

\maketitle

\section{Introduction}

A standard result in Banach space theory, due to Bessaga and Pe\l czy\'nski, \cite{BP}, is that every seminormalized, weakly null sequence admits a basic subsequence with basis constant arbitrarily close to $1$.  Maurey and Rosenthal \cite{MR} famously gave an example of a normalized, weakly null sequence admitting no unconditional subsequence.  Since then, significant attention has been paid to notions of partial unconditionality of subsequences of weakly null sequences.  Two examples of such notions, examples which will be the focus of this work, are as follows: Given a seminormalized, weakly null sequence $(x_n)$, for any $\ee>0$ and $q\in \nn$, there exists a subsequence $(x_{n_i})$ of $(x_n)$ so that for any scalars $(a_i)\in c_{00}$, the scalar sequences with finite support, and any set $E\subset \nn$ with $|E|\leqslant q$, $$\|\sum_{i\in E} a_i x_{n_i}\|\leqslant (1+\ee)\|\sum a_ix_{n_i}\|.$$  The first proof of this fact was given by Odell \cite{O}, wherein he applied a standard diagonalization procedure to obtain a subsequence $(x_{n_i})$ of $(x_n)$ so that for any scalar sequence $(a_i)\in c_{00}$ and any set $E\subset \nn$ with $|E|\leqslant \min E$, $$\|\sum_{i\in E}a_i x_{n_i}\|\leqslant (2+\ee)\|\sum a_i x_{n_i}\|.$$  The latter property of the sequence $(x_{n_i})$ is called \emph{Schreier unconditionality}.  

Another mode of unconditionality was introduced by Elton \cite{E}, and it is known as \emph{Elton unconditionality} or \emph{near unconditionality}.  For $0<\delta <1$, a sequence $(x_n)$ is $\delta$- Elton unconditional provided that there exists a constant $K\geqslant 1$ such that for any $(a_n)\in c_{00}$ with $\sup_n |a_n|\leqslant 1$ and any $E\subset \{n: |a_n|\geqslant \delta\}$, $$\|\sum_{i\in E} a_i x_i\|\leqslant K\|\sum a_ix_i\|.$$  Elton \cite{E} showed that for each $\delta\in (0,1)$, every seminormalized, weakly null sequence admits a subsequence which is $\delta$-Elton unconditional with constant $K=K(\delta)$ depending only on $\delta$.   

The focus of this work is to combine Schreier unconditionality with monotonicity to show that
every semi-normalized weakly null sequence admits a subsequence
for which certain projections have norm 
arbitrarily close to $1$. An analogous result which combines Elton unconditionality and monotonicity is proved for a certain class of weakly null sequences.
 
 As an application we show that the members of this class of sequences admit quasi-greedy subsequences with quasi-greedy constant arbitrarily close to $1$,
which improves \cite[Theorem 5.4]{DKK}.

We are ready to state the main theorem.  

\begin{theorem} Let $(x_n)$ be a seminormalized, weakly null sequence.  \begin{enumerate}[(i)]\item For any $\ee>0$ and any sequence $(q_i)_{i\geqslant 0}$ of natural numbers, there exists a subsequence $(y_n)$ of $(x_n)$ so that for any $n\in \nn$, any set $E\subset \nn$ with $n<\min E$ and $|E|\leqslant q_n$, and any $(a_i)\in c_{00}$,  $$\|\sum_{i=1}^n a_i y_i+ \sum_{i\in E} a_i y_i\|\leqslant (1+\ee) \|\sum a_iy_i\|.$$  \item Suppose that no subsequence of $(x_n)$ generates a spreading model equivalent to the canonical $c_0$ basis.  For any $\ee>0$ and sequence $(\delta_n)_{n\geqslant 0}\subset (0,1)$, there exists a subsequence $(y_n)$ of $(x_n)$ so that for any scalars $(a_i)\in c_{00}$ such that $\sup_{i\in \nn}|a_i|\leqslant 1$, any $n\in \nn$, and any set $E\subset \{i\in \nn: n<i, |a_i|\geqslant \delta_n\}$, $$\|\sum_{i=1}^n a_iy_i + \sum_{i\in E} a_iy_i\|\leqslant (1+\ee)\|\sum a_iy_i\|.$$ \end{enumerate}

\label{main theorem}
\end{theorem}

In this work, $X$ will denote a Banach space over the real or complex scalars and $B_X$ will denote its closed unit ball. We  let $\mathbb{B}$ denote $[-1,1]$ in the real case or the closed  unit disk in the complex plane in the  complex case.  For each $n\in \nn$, $\mathbb{B}^n$ will be endowed with the $\ell_\infty^n$ metric.  If $N$ is an infinite subset of $\nn$, we let $[N]$ (resp. $[N]^n$) denote the infinite (resp. cardinality $n$) subsets of $N$.  We identify subsets of $\nn$ with strictly increasing sequences in $\nn$ in the natural way and let $\nn_0:=\nn\cup \{0\}$.

\section{Extensions of Schreier and Elton unconditionality}

Recall that a sequence $(x_n)$ in a Banach space has the sequence $(s_n)$ in a (possibly different) Banach space as a \emph{spreading model} provided that for each $k\in \nn$ and $\ee>0$, there exists $m=m(k, \ee)$ so that for any $m\leqslant n_1<\ldots < n_k$ and any scalars $(a_i)_{i=1}^k$, $$\Bigl|\|\sum_{i=1}^k a_i s_i\| - \|\sum_{i=1}^k a_i x_{n_i}\|\Bigr|<\ee.$$  We say $(x_n)$ \emph{generates} the spreading model $(s_n)$.  Recall also that every seminormalized, weakly null sequence has a subsequence which generates a $1$-suppression unconditional spreading model.

Theorem \ref{main theorem}$(i)$ will follow from the next lemma.

\begin{lemma} Let $X$ be a Banach space.  Given a weakly null sequence $(x_n)\subset B_X$, $(q_n)_{n\geqslant 0} \subset \nn$, and $(\ee_n)_{n\geqslant 0}\subset (0,\infty)$, there exist $1\leqslant n_1<n_2<\ldots$ with the following property: For any $k\in \nn_0$, $x^*_0\in B_{X^*}$, and any $\varnothing \neq E\subset \nn$ with $k<\min E$ and $|E|\leqslant q_k$, there exists a functional $x^*\in B_{X^*}$ so that \begin{enumerate}[(i)]\item for $i\in \{1, \ldots, k\}\cup E$, $|x^*_0(x_{n_i})- x^*(x_{n_i})|\leqslant \ee_k$, \item for $1\leqslant i\leqslant \max E$, $i\notin \{1, \ldots, k\}\cup E$, $|x^*(x_{n_i})|\leqslant \ee_i.$ \end{enumerate}

\label{lemma1}
\end{lemma}

For clarity, we isolate the following step of the proof of Lemma \ref{lemma1}.   

\begin{lemma} Suppose $k\in \nn_0$, $(y_i)_{i=1}^k\subset B_X$, and $(z_i)\subset B_X$ is weakly null. Fix $l\in \nn$ and $\ee>0$. If $A\subset \mathbb{B}^k$ and $B\subset \mathbb{B}^l$ are any sets, then there exists a subsequence $(z_i)_{i\in M}$ of $(z_i)$ so that if $x^*\in B_{X^*}$ and $r_0<\ldots < r_l$, $r_i\in M$, are such that $(x^*(y_i))_{i=1}^k\in A$ and $(x^*(z_{r_i}))_{i=1}^l\in B$, then there exists $y^*\in B_{X^*}$ so that $(y^*(y_i))_{i=1}^k\in A$, $(y^*(z_{r_i}))_{i=1}^l\in B$, and $|y^*(z_{r_0})|\leqslant \ee$.  
\label{lemma2}
\end{lemma}

Here it should be understood that if $k=0$, the hypotheses and conclusions involving $(y_i)_{i=1}^k$ and $A$ should be omitted.  

The next lemma will easily yield Theorem \ref{main theorem}$(ii)$ after a standard diagonalization.  

\begin{lemma} Fix $n\in \nn_0$, $\ee>0$, and $\delta>0$.  If $(x_i)_{i=1}^n\subset X$ and $(y_i)\subset X$ is a semi-normalized, weakly null sequence generating a spreading model not equivalent to the canonical $c_0$ basis,  then there exists a subsequence $(z_i)$ of $(y_i)$ so that for all scalars $(a_i)_{i=1}^n$ and $(b_i)\in c_{00}$ with $|a_i|, |b_i|\leqslant 1$ for all appropriate $i$, and for all sets $E\subset \{i\in \nn: |b_i|\geqslant \delta\}$, $$ \|\sum_{i=1}^n a_ix_i + \sum_{i\in E} b_i z_i\|\leqslant (1+\ee) \|\sum_{i=1}^n a_ix_i + \sum b_iz_i\|.$$

\label{lemma3}
\end{lemma}

We use these lemmas to prove Theorem \ref{main theorem}, and then return to the proofs of the lemmas.

\begin{proof}[Proof of Theorem \ref{main theorem}]$(i)$ Let us assume $\ee<1$ and $(x_i)\subset B_X$ is $(1+\ee/2)$-basic. Fix $\eta>0$ such that $0<\eta^{-1}< \inf_n \|x_n\|$.  By \cite{O}, by passing to a subsequence of $(x_i)$, we may assume that for any $(a_i)\in c_{00}$, $\|(a_i)\|_{c_0}\leqslant \eta \|\sum a_ix_i\|$.   Fix $(\ee_k)_{k\geqslant 0}$ so that for each $k\in \nn_0$, $$ (k+q_k)\ee_k + \sum_{i=k+1}^\infty \ee_i < \ee/2\eta.$$ Without loss of generality, by Lemma \ref{lemma1} we may pass to a subsequence, relabel, and assume $(x_i)$ satisfies the conclusions of Lemma \ref{lemma1} with this choice of $(\ee_k)_{k\geqslant 0}$ and $(q_k)_{k\geqslant 0}$.  Choose $k\in \nn_0$, $E\subset \nn$ with $k<\min E$ and $|E|\leqslant q_k$, and $(a_i)\in c_{00}$.  If $E=\varnothing$, the fact that $(x_i)$ is $(1+\ee/2)$-basic gives the desired inequality.  So assume $E\neq \varnothing$.  Fix $x^*_0\in B_{X^*}$ so that $$\|\sum_{i=1}^k a_ix_i + \sum_{i\in E}a_ix_i\|=x^*_0\Bigl(\sum_{i=1}^k a_ix_i + \sum_{i\in E}a_ix_i\Bigr).$$  Choose $x^*$ as in the conclusion of Lemma \ref{lemma1}.  Then \begin{align*} (1+\ee/2)\|\sum a_ix_i\| & \geqslant \|\sum_{i=1}^{\max E}a_ix_i\|  \geqslant x^*\Bigl(\sum_{i=1}^{\max E}a_ix_i\Bigr) \\ & \geqslant x^*_0\Bigl(\sum_{i=1}^k a_ix_i + \sum_{i\in E}a_ix_i\Bigr) - \sum_{i=1}^k|a_i||(x^*_0-x^*)(x_i)|  \\ & - \sum_{i\in E}|a_i||(x^*_0-x^*)(x_i)| - \sum_{\underset{i\notin \{1, \ldots, k\}\cup E}{i=1}}^{\max E} |a_i||x^*(x_i)| \\ & \geqslant \|\sum_{i=1}^k a_ix_i + \sum_{i\in E}a_ix_i\| - \|(a_i)\|_{c_0}\Bigl((k + q_k)\ee_k - \sum_{i=k+1}^\infty \ee_i\Bigr) \\ & \geqslant \|\sum_{i=1}^k a_ix_i + \sum_{i\in E}a_ix_i\| - \eta\Bigl( (k+q_k)\ee_k + \sum_{i=k+1}^\infty \ee_i\Bigr)\|\sum a_ix_i\| \\ & \geqslant \|\sum_{i=1}^k a_ix_i + \sum_{i\in E}a_ix_i\| - (\ee/2)\|\sum a_ix_i\|. \end{align*}

Adding $(\ee/2)\|\sum a_ix_i\|$ to both sides finishes the proof.  
 
$(ii)$ This follows easily from recursive applications of Lemma \ref{lemma3}.  

\end{proof}

\begin{proof}[Proof of Lemma \ref{lemma2}] 

Let $\uuu$ consist of all $E=(j_i)_{i=1}^l\in [\nn]^l$ such that there exists $x^*\in B_{X^*}$ so that $(x^*(y_i))_{i=1}^k\in A$ and $(x^*(z_{j_i}))_{i=1}^l\in B$.  By the finite Ramsey theorem, there exists $N\in \infin$ such that either $\uuu\cap [N]^l=\varnothing$ or $[N]^l\subset \uuu$.  In the first case, we let $M=N$ and note that the conclusion is vacuously satisfied.  In the second case, let $\vvv$ consist of those $E=(j_i)_{i=0}^l\in [N]^{l+1}$ such that there exists $x^*\in B_{X^*}$ such that $(x^*(y_i))_{i=1}^k\in A$, $(x^*(z_{j_i}))_{i=1}^l\in B$, and $|x^*(z_{j_0})|\leqslant \ee$.  Applying the finite Ramsey theorem again, there exists $M\in [N]$ such that either $[M]^{l+1}\cap \vvv=\varnothing$ or $[M]^{l+1}\subset \vvv$.  We show that the first alternative cannot hold, so that this $M$ satisfies the conclusion.  Assume that the first alternative holds.  Let $M=(m_i)_{i=1}^\infty$, $m_1<m_2<\ldots$.  Note that for each $1\leqslant  j\in \nn$, $(m_{j+r})_{r=1}^l\in \uuu$.  Therefore there exists $x^*_j\in B_{X^*}$ so that $(x^*_j(y_i))_{i=1}^k\in A$ and $(x^*_j(z_{m_{j+r}}))_{r=1}^l\in B$.  Since for each $1\leqslant i\leqslant j$, $(m_i, m_{j+1}, \ldots, m_{j+l})\notin \vvv$, it must be that $|x^*_j(x_{m_i})|\geqslant \ee$.  Then if $x^*$ is any $w^*$ cluster point of $(x_j^*)$, $|x^*(z_{m_i})|\geqslant \ee$ for all $i\in \nn$, contradicting the weak nullity of $(z_i)$.

\end{proof}

\begin{proof}[Proof of Lemma \ref{lemma1}] Fix $(\delta_n)_{n=0}^\infty \subset (0,\infty)$ decreasing to zero such that $2\sum_{i=k}^\infty\delta_i \leqslant \ee_k$ for each $k\in \nn_0$.   Use Lemma \ref{lemma2} recursively as follows: Assume $1\leqslant n_1<\ldots <n_i$ and $\nn=M_0\supset \ldots \supset M_i\in [\nn]$ have been chosen.  For fixed $1\leqslant j\leqslant q_{i+1}$, let $(A_p, B_p)_{p=1}^r$ be such that $(A_p)_{p=1}^r$ is a cover of $\mathbb{B}^i$ by sets of diameter less than $\delta_{i+1}$ and $(B_p)_{p=1}^r$ is a cover of $\mathbb{B}^j$ by sets of diameter less than $\delta_{i+1}$.  By applying Lemma \ref{lemma2} successively to $A=A_p$ and $B=B_q$ for each $1\leqslant p,q\leqslant r$, we may pass to a subsequence $M^j_{i+1}$ of $M_i$ satisfying the conclusion,  with $\varepsilon=\delta_{i+1}$, of Lemma \ref{lemma2} for each of these pairs.  Of course, we may do this successively for each $1\leqslant j\leqslant q_{i+1}$.  We let $M$ be the sequence obtained by repeating this process for each $1\leqslant j\leqslant q_{i+1}$, $n_{i+1}=\min M$, and $M_{i+1}=M\setminus \{n_{i+1}\}$.  We will show that the subsequence $(x_{n_i})$ resulting from this recursive definition satisfies the conclusion of Lemma \ref{lemma1}.

Fix $k\in \nn_0$, $\varnothing \neq E\subset \nn$ with $k<\min E$ and $|E|\leqslant q_k$.   Fix $x^*_0\in B_{X^*}$.  We choose $(x_i^*)_{i=1}^{\max E}$ recursively as follows: If $x_i^*\in B_{X^*}$ has been chosen, and if $i+1\in \{1, \ldots k\}\cup E$, let $x_{i+1}^*=x_i^*$.  Otherwise, we let $A= A_p\subset \mathbb{B}^i$ and $B=B_q\subset \mathbb{B}^j$ be members of the covers of $\mathbb{B}^i$ and $\mathbb{B}^j$, respectively, from the previous paragraph such that $(x_i^*(x_{n_r}))_{r=1}^i\in A$ and $(x^*_i(x_{n_r}))_{i<r\in E}\in B$, where $j=|E\cap \{i+1, i+2, \ldots\}|$.  Then there exists $x_{i+1}^*\in B_{X^*}$ so that $(x^*_{i+1}(x_{n_r}))_{i=1}^i\in A$, $(x^*_{i+1}(x_{n_r}))_{i<r\in E}\in B$, and $|x^*_{i+1}(x_{n_{i+1}})|\leqslant \delta_{i+1}$. This completes the recursive definition.  Observe that with this definition, for each $1\leqslant i<\max E$ and $j\in \{1, \ldots, i\}\cup (E\cap \{i+1, i+2, \ldots\})$, $|x^*_i(x_{n_j})- x^*_{i+1}(x_{n_j})|\leqslant \delta_{i+1}$.  Then for $j\in \{1, \ldots, k\}\cup E$, $$|(x^*_{\max E}- x^*_0)(x_{n_j})| = |(x^*_{\max E}- x_k^*)(x_{n_j})|\leqslant \sum_{i=k}^{\max E-1}|(x^*_{i+1}- x^*_i)(x_{n_j})|\leqslant \sum_{i=k}^\infty\delta_i \leqslant \ee_k.$$  For $j\leqslant \max E$ with $j\notin \{1, \ldots, k\}\cup E$, $$|x^*_{\max E}(x_{n_j})|\leqslant |x^*_j(x_{n_j})| + \sum_{i=j+1}^{\max E-1} |(x^*_{i+1}-x_i^*)(x_{n_j})|\leqslant \delta_j+\sum_{i=j}^\infty \delta_i \leqslant \ee_j.$$

\end{proof}

\begin{proof}[Proof of Lemma \ref{lemma3}] Assume $\ee<1$.  Fix $\alpha\in (0, 1)$ so that $(1+2\alpha)/(1-\alpha)<1+\ee$. Let $M=\max_{1\leqslant i\leqslant n} \|x_i\|$. By passing to a subsequence, we may assume $(y_i)$ generates the spreading model $(s_i)$, and recall that $(s_i)$ is not equivalent to the canonical $c_0$ basis. Let $g(n)=\|\sum_{i=1}^n s_i\|$.  Note that the properties of spreading models generated by weakly null sequences imply that $g(n)=\|\sum_{i\in A}s_i\|$ for any $A$ with $|A|=n$ and $g(1)\leqslant g(2)\leqslant \ldots$.   Since $(s_i)$ is not equivalent to the $c_0$ basis, $g(n)\to \infty$.  Choose $q\in \nn$ so that $\delta g(q)/8 > nM/\alpha$. By passing to a subsequence of $(y_i)$ and relabelling, we may assume that for any set $B\subset \nn$ with $|B|\leqslant q$, $(y_i)_{i\in B}$ is $2$-equivalent to $(s_i)_{i=1}^{|B|}$.  By \cite[Theorem 3.3]{DOSZ}, we may  assume by passing to a further subsequence that
  $(y_i)$ is $\delta$-Elton unconditional with constant $1+\alpha$. By passing to yet a further subsequence and arguing as in Theorem~ \ref{main theorem}(i), we may assume that for any set $B$ with $|B|\leqslant q$ and any scalars $(a_i)_{i=1}^n$, $(b_i)\in c_{00}$, $\|\sum_{i=1}^n a_ix_i +\sum_{i\in B} b_iy_i\|\leqslant (1+\ee)\|\sum_{i=1}^n a_ix_i+ \sum_i b_iy_i\|$.   

  Fix scalars $(a_i)_{i=1}^n$ and $(b_i)\in c_{00}$ with $\max_i |a_i|, \max_i |b_i|\leqslant 1$.  We consider two cases. For the first case, assume that $\| \sum_i b_iy_i\|\leqslant nM/\alpha$,  and let $A=\{i: |b_i|\geqslant \delta\}$.  We claim that $|A|<q$.  If it were not so, we could choose $B\subset A$ with $|B|=q$.  Then by our choice of $(y_i)$ and the $1$-suppression unconditionality of $(s_i)$,  \begin{align*} nM/\alpha & \geqslant \|\sum b_i y_i\| \geqslant (1/2) \|\sum_{i\in B} b_i y_i \| \\ & \geqslant (1/4)\|\sum_{i\in B} b_i s_i\| \geqslant (\delta/8) g(q),\end{align*} a contradiction.  Since $|A|<q$, $\|\sum_{i=1}^n a_ix_i+ \sum_{i\in A} b_iy_i\|\leqslant (1+\ee)\|\sum_{i=1}^n a_ix_i+\sum b_iy_i\|$ by the last sentence of the previous paragraph.

For the second case, assume that $\| \sum_i b_iy_i\|\geqslant nM/\alpha$, and note that then $$\|\sum_{i=1}^n a_ix_i + \sum_i b_iy_i\| \geqslant \|\sum b_iy_i\| - nM \geqslant (1-\alpha)\|\sum_i b_iy_i\|.$$  Hence \begin{align*} \|\sum_{i=1}^n a_ix_i + \sum_{i\in E} b_i y_i\| & \leqslant nM + \|\sum_{i\in E} b_i y_i\| \leqslant (1+2\alpha) \|\sum_i b_i y_i\| \\ & \leqslant (1+2\alpha)/(1-\alpha)\|\sum_{i=1}^n a_ix_i + \sum_i b_iy_i\| \\ & \leqslant (1+\ee)\|\sum_{i=1}^n a_ix_i + \sum_i b_iy_i\|. \end{align*}

\end{proof}

\section{Applications to quasi-greedy bases}

Fix a seminormalized basis $(e_i)$ for a Banach space $X$.  For $n\in \nn$, $x=\sum a_ie_i$, and $n\in \nn$, we let $G_n(x)=\sum_{i\in A} a_ie_i$, where $|A|=n$ and $\min_{i\in A} |a_i|\geqslant \max_{i\in \nn\setminus A}|a_i|$.  Of course, such a set $A$ will not be uniquely defined in general.  Recall that the basis $(e_i)$ is said to be \emph{quasi-greedy} provided there exists a constant $C_{qg}$ such that for all $x\in X$ and $n\in \nn$, $\|G_n(x)\|\leqslant C_{qg} \|x\|$. This definition was introduced by Konyagin and Temlyakov \cite{KT}
 and was shown to be equivalent to norm convergence of $(G_n(x))$ to $x$, for each $x \in X$,  by Wojtaszczyk \cite{W}. In order to check that $(e_i)$ is quasi-greedy with constant $C$ it clearly suffices to prove that for each $a>0$ and $x\in X$, $\|\mathcal{G}_a(x)\|\leqslant C\|x\|$, where $\mathcal{G}_a(x)=\sum_{i:|a_i|\geqslant a} a_ie_i$ and $x=\sum a_ie_i$.

\begin{theorem} Let $(x_n)$ be a seminormalized, weakly null sequence so that no subsequence generates a spreading model equivalent to the canonical $c_0$ basis.  Then for any $\ee>0$, there exists a subsequence $(y_n)$ of $(x_n)$ which is a quasi-greedy basis for its span with constant $(1+\ee)$.

\end{theorem}

\begin{remarks} (a) The weaker result  with $1+\ee$ replaced by $3+\ee$ was  
shown in \cite[Theorem 5.4]{DKK} (cf.\ \cite[Corollary 3.4]{DOSZ}).

 (b) The result fails if we omit the hypothesis concerning $c_0$ spreading models.  In fact, in \cite[p. 59]{DOSZ}, an equivalent norm on $c_0$ is given so that the unit vector basis has no quasi-greedy subsequence with constant less than $8/7$. 

(c) The constant of $1 + \varepsilon$  is optimal, even in a Hilbert space. Indeed, it  was proved in \cite[Corollary 2.3]{GW} that a normalized basis $(x_n)$ of a Hilbert space  is quasi-greedy with $C_{qg}=1$ if and only if $(x_n)$ is
orthonormal. Hence, if $\langle x_n, x_m \rangle > 0$ for all $n,m$, it follows that $(x_n)$ does not admit a subsequence with $C_{qg}=1$. Clearly, such an $(x_n)$  will also satisfy the spreading model assumption of the theorem. 

(d) It was proved recently \cite{AA} that a semi-normalized basis $(x_n)$ of a Banach space is quasi-greedy with $C_{qg}=1$ if and only if 
$(x_n)$ is $1$-suppression unconditional.

\end{remarks}

\begin{proof} Assume $\ee<1$. By passing to a subsequence, scaling, and passing to a closely equivalent norm, we may assume $(x_n)$ is normalized and monotone. As above, we may assume $(x_n)$ generates a spreading model $(s_n)$, and we let $g(n)=\|\sum_{i=1}^n s_i\|$, $g(0)=0$.   Choose $(q_n)_{n=0}^\infty \subset \nn$ so that $g(q_k)\ee >32 (k+1)$ for all $k\in \nn_0$.  By passing to a subsequence, we may assume that for any $k\in \nn_0$ and any $\varnothing \neq C\subset \nn$ with $k< \min C$ and $|C|\leqslant q_k$, $(x_n)_{n\in C}$ is $2$-equivalent to $(s_n)_{n\in B}$.     Let $(y_i)=(x_{n_i})$, where $(x_{n_i})$ satisfies the conclusion of Theorem \ref{main theorem}$(i)$ with this choice of $(q_n)_{n=0}^\infty$ and $\ee$ replaced by $\ee/2$.  

Fix $x=\sum a_iy_i$ with $\|x\|=1$ and fix $a>0$.  Choose $k\in \nn_0$ so that $2a\in [\ee/(k+1), \ee/k)$, where $\ee/0:=\infty$.   Write $G_a(x)= \sum_{i\in A}a_iy_i + \sum_{i\in B} a_iy_i$, where $A=\{i:i\leqslant k, |a_i|\geqslant a\}$ and $B= \{i: i> k, |a_i|\geqslant a\}$.  We claim that $|B|\leqslant q_k$.   If it were not so, we could choose $C\subset B$ with $|C|=q_k$.  Then \begin{align*} 1 & = \|x\| \geqslant \frac{1}{2} \|\sum_{i=1}^k a_i y_i + \sum_{i\in C}a_i y_i\| \\ & \geqslant \frac{1}{4} \|\sum_{i\in C} a_i y_i\| \geqslant \frac{1}{8} \|\sum_{i\in C} a_i s_i\| \\ & \geqslant \frac{a g(q_k)}{16}  \geqslant \frac{g(q_k) \ee}{32 (k+1)}>1, \end{align*} a contradiction.  Thus $|B|\leqslant q_k$, as claimed.  Then \begin{align*} \|G_a(x)\| & = \|\sum_{i=1}^k a_i y_i + \sum_{i\in B} a_i y_i - \sum_{i\in \{1, \ldots, k\}\setminus A} a_i y_i\|  \leqslant (1+\ee/2) \|x\| + ka \\ & \leqslant 1+\ee/2 +\ee/2 = 1+\ee. \end{align*}

\end{proof}

\begin{remark}

Recall that for $\Delta\geqslant 1$, a basic sequence $(e_i)$ is \emph{democratic} if for any finite sets $A, B\subset \nn$ with $|A|\leqslant |B|$, $\|\sum_{i\in A} e_i\|\leqslant \Delta \|\sum_{i\in B}e_i\|$.  It was shown in \cite[Proposition 5.3]{DKK} that any seminormalized, weakly null sequence having no subsequence generating a spreading model equivalent to the canonical $c_0$ basis has a $(1+\ee)$ democratic subsequence.  Recall also that a basic sequence $(e_i)$ is said to be \emph{almost greedy} provided there exists a constant $C_a$ so that for any $x\in X$, $\|x-G_n(x)\|\leqslant C_a \tilde{\sigma}_n(x)$, where $$\tilde{\sigma}_n(x)= \inf_{|A|\leqslant n} \|x-\sum_{i\in A} e_i^*(x)e_i\|.$$  It was also shown in \cite[Theorem 3.3]{DKKT} that a basic sequence is almost greedy if and only if it is quasi-greedy and democratic, and the constant $C_a$ of $(e_i)$ does not exceed $8C^4 \Delta + C+1$ if $(e_i)$ is quasi-greedy with constant $C$ and democratic with constant $\Delta$.  Thus we have shown that any seminormalized, weakly null sequence having no subsequence generating a spreading model equivalent to the canonical $c_0$ basis has a subsequence which is almost greedy with constant $10+\ee$.

\end{remark}

\end{document}